\documentclass[12pt]{amsart}
\usepackage{amssymb}
\usepackage{amsfonts}
\usepackage{amssymb,latexsym}
\usepackage{enumerate}
\usepackage{url}
\usepackage{mathrsfs}

\makeatletter
\@namedef{subjclassname@2010}{%
  \textup{2010} Mathematics Subject Classification}
\makeatother

  [2002/01/19 v2.2g %
    AMS font definitions%
  ]
\DeclareFontFamily{U}{euf}{}
\DeclareFontShape{U}{euf}{m}{n}{%
  <5><6><7><8><9>gen*eufm%
  <10><10.95><12><14.4><17.28><20.74><24.88>eufm10%
  }{}
\DeclareFontShape{U}{euf}{b}{n}{%
  <5><6><7><8><9>gen*eufb%
  <10><10.95><12><14.4><17.28><20.74><24.88>eufb10%
  }{}

  [2002/01/19 v2.2g %
    AMS font definitions%
  ]
\DeclareFontFamily{U}{msb}{}
\DeclareFontShape{U}{msb}{m}{n}{%
  <5><6><7><8><9>gen*msbm%
  <10><10.95><12><14.4><17.28><20.74><24.88>msbm10%
  }{}

  [2002/01/19 v2.2g %
    AMS font definitions%
  ]
\DeclareFontFamily{U}{msa}{}
\DeclareFontShape{U}{msa}{m}{n}{%
  <5><6><7><8><9>gen*msam%
  <10><10.95><12><14.4><17.28><20.74><24.88>msam10%
  }{}

\newtheorem{theorem}{Theorem}[section]
\newtheorem{lemma}[theorem]{Lemma}
\newtheorem{proposition}[theorem]{Proposition}
\newtheorem{corollary}[theorem]{Corollary}
\newtheorem{remark}[theorem]{Remark}

\theoremstyle{definition}

\numberwithin{equation}{section} 

\begin{document}

\title[Dirichlet's lambda function]
{On Dirichlet's lambda function}



\author{Su Hu}
\address{Department of Mathematics, South China University of Technology, Guangzhou, Guangdong 510640, China}
\email{mahusu@scut.edu.cn}

\author{Min-Soo Kim}
\address{Division of Mathematics, Science, and Computers, Kyungnam University, 7(Woryeong-dong) kyungnamdaehak-ro, Masanhappo-gu, Changwon-si,
Gyeongsangnam-do 51767, Republic of Korea}
\email{mskim@kyungnam.ac.kr}




\begin{abstract}
Let $$\lambda(s)=\sum_{n=0}^\infty\frac1{(2n+1)^s},$$
$$\beta(s)=\sum_{n=0}^\infty\frac{(-1)^{n}}{(2n+1)^s},$$
and
$$\eta(s)=\sum_{n=1}^\infty\frac{(-1)^{n-1}}{n^s}$$
be the Dirichlet lambda function,
its alternating form, and
the Dirichlet eta function, respectively.
According to a recent historical book by Varadarajan (\cite[p.~70]{Varadarajan}), these three functions were investigated by Euler under the notations $N(s)$, $L(s)$, and $M(s)$, respectively.

In this paper, we shall present some additional properties for them. That is, we obtain a number of infinite families of linear recurrence relations for $\lambda(s)$ at positive even integer arguments $\lambda(2m)$,
convolution identities for special values of $\lambda(s)$ at even arguments and special values
of $\beta(s)$ at odd arguments, and a power series expansion for the alternating Hurwitz zeta function  $J(s,a)$, which involves a known
one for $\eta(s)$. \end{abstract}
\maketitle


\section{Introduction}

\subsection{Background} According to Varadarajan \cite[p.~59]{Varadarajan}, Pietro Mengoli (1625--1686) posed the problem of finding the sum of
the series \begin{equation}\label{Mengoli}
\sum_{n=1}^{\infty}\frac{1}{n^{2}}=1+\frac{1}{4}+\frac{1}{9}+\cdots.
\end{equation}
This problem was first solved by Euler, who communicated his findings in a letter to Daniel Bernoulli, asserting that
 \begin{equation}\label{Mengoli}
1+\frac{1}{4}+\frac{1}{9}+\cdots=\frac{\pi^{2}}{6},
\end{equation}
more generally, let $$\zeta(s)=\sum_{n=1}^{\infty}\frac{1}{n^{s}}$$ be Riemann's zeta function. Then,
Euler proved the following formula:
\begin{equation}\label{Riemann}
\zeta(2m)=1+\frac{1}{2^{2m}}+\frac{1}{3^{2m}}+\cdots=\frac{(-1)^{m-1}B_{2m}2^{2m}}{2(2m)!}\pi^{2m},
\end{equation}
where the $B_{2m}$ are the Bernoulli numbers, defined by the generating function
\begin{equation}\label{Bernoulli}
\frac{t}{e^{t}-1}=1-\frac{t}{2}+\sum_{m=1}^{\infty}B_{2m}\frac{t^{2m}}{(2m)!}.
\end{equation}
One of Euler's proofs of (\ref{Mengoli}) from around 1742 is based on the following infinite product expansion
of $\frac{\sin x}{x}$:
\begin{equation}\label{product}
\frac{\sin x}{x}=\prod_{n=1}^{\infty}\left(1-\frac{x^{2}}{n^{2}\pi^{2}}\right).
\end{equation}
Euler also provided another proof of (\ref{Mengoli}) by starting from the formula
\begin{equation}\label{Euler2}
\frac{1}{2}(\arcsin x)^{2}=\int_{0}^{x}\frac{\arcsin t}{\sqrt{1-t^{2}}}dt.
\end{equation} First by taking $x=1$ in the left-hand side, we get $\frac{\pi^{2}}{8}$.
By expanding $\arcsin t$ as a power series and integrating term-by-term on the right-hand side, we obtain the following summation:
$$1+\frac{1}{3^{2}}+\frac{1}{5^{2}}+\cdots=\sum_{n=0}^{\infty}\frac{1}{(2n+1)^{2}}.$$
After comparing the results on both sides, we have
\begin{equation}\label{Euler3}
\sum_{n=0}^{\infty}\frac{1}{(2n+1)^{2}}=\frac{\pi^{2}}{8}.
\end{equation}
Then, by noticing that
\begin{equation}\label{Euler4}
1+\frac{1}{3^{2}}+\frac{1}{5^{2}}+\cdots=\zeta(2)-\frac{1}{2^{2}}\zeta(2)=\frac{3}{4}\zeta(2),
\end{equation}
we recover (\ref{Mengoli}) (see \cite[pp.~62--63]{Varadarajan}).

According to Abramowitz and Stegun's handbook \cite[pp.~807--808]{AS}, the function
 \begin{equation}\label{lam-def}
\lambda(s)=\sum_{n=0}^\infty\frac1{(2n+1)^s}=(1-2^{-s})\zeta(s),\quad\text{Re}(s)>1,
\end{equation}
is usually named the Dirichlet lambda function, and was also studied by Euler under the notation $N(s)$ (see \cite[p.~70]{Varadarajan}).
In addition, (\ref{Euler3}) provided the special value of $\lambda(s)$ at 2, and the special values of $\lambda(s)$ at any even positive integer $2m$
are calculated by the following formula:
 \begin{equation}\label{lam-ft-eq}
\lambda(2m)=(-1)^{m}\frac{\pi^{2m}}{4(2m-1)!}E_{2m-1}(0),\quad m\geq 1,
\end{equation}
where the Euler polynomials $E_{n}(x)$ are defined by the following generating function:
\begin{equation}\label{euler poly-def}
\frac{2}{e^t+1}e^{tx}=\sum_{n=0}^\infty E_n(x)\frac{t^n}{n!}
\end{equation}
(see \cite[(4.16)]{Di} and \cite[(2.15)]{HKK}).
The integers $E_{n}=2^{n}E_{n}\left({1}/{2}\right),n\geq0,$ are called Euler numbers,
and can also be defined as the coefficients
of $t^n/n!$ in the Taylor expansion of $\textrm{sech}(t),~ |t|<\pi/2.$
For example, $E_0=1,E_2=-1,E_4=5,$ and $E_6=-61.$ Euler numbers and polynomials (so called by Scherk in 1825) were introduced in Euler's famous book,
Insitutiones Calculi Differentials (1755, pp.~487--491 and p.~522).

Recently, the Euler polynomials have been employed in several different applications, such as semi-classical approximations
of quantum probability distributions (see \cite{BM}), and various approximations and
expansion formulas in discrete mathematics and number theory (see \cite{AS,PR}).
Euler polynomials can be defined using various methods, depending on their applications (see \cite{GR,sun}).
For example, the explicit formula
\begin{equation}\label{ex-form}
E_n(x)=\sum_{k=0}^n\binom nkx^{n-k}E_k(0)
\end{equation}
shows that $E_n(x)$ is a polynomial of degree $n.$ There are various well-known approaches to the theory of Euler polynomials,
and these can be found in the classical papers by Euler \cite{Euler}, N\"orlund \cite{No}, and Raabe \cite{Ra}.

 It may be interesting to note that there is also a connection between the generalized Euler numbers and the ideal class group of the
$p^{n+1}$-th cyclotomic field when $p$ is a prime number.
For details, we refer the reader to the recent paper~\cite{HK-I}, especially~\cite[Proposition 3.4]{HK-I}.

In addition, Euler also studied the function
\begin{equation}\label{beta-def}
\beta(s)=\sum_{n=0}^\infty\frac{(-1)^{n}}{(2n+1)^s},\quad\text{Re}(s)>0,
\end{equation} ($L(s)$ in his notation: see \cite[p.~70]{Varadarajan}). This is the alternating form of the Dirichlet lambda function $\lambda(s)$ (\ref{lam-def}).
Furthermore, the constant $\beta(2)= G$ is usually named Catalan's constant (see \cite{RZ}, \cite{WG}, \cite[p.~807]{AS}, and \cite[p.~53]{Fi}).
It is known that the special values of $\beta(s)$ at odd positive integers $2m+1$ are given by
\begin{equation}\label{beta-re}
\beta(2m+1)=(-1)^n\frac{E_{2m}}{2(2m)!}\left(\frac\pi2\right)^{2m+1},
\end{equation}
where $E_{2m}$ are the Euler numbers,
and $\beta(s)$ admits the following interesting integral representation using a trigonometric function:  $$\beta(s)=\frac{1}{2\Gamma(s)}\int_0^\infty\frac{t^{s-1}}{\text{cosh}(t)},\quad\text{Re}(s)>0,$$
where $\Gamma$ denotes the gamma function (see \cite[p.~56]{Fi}).

\subsection{Main results}
Dating back to Euler, the study of the special values of the Riemann zeta function at positive
integers has a long history; hence there is a vast collection of related literature. Here, we mention some that are directly related to our study.

 In 1987, Song \cite{So} found the following linear recurrence relation for Riemann's zeta function at even arguments $\zeta(2m)$, using the Fourier series expansion of periodic functions:
$$(-1)^{m+1}\frac{\pi^{2m}\cdot m}{(2m+1)!}+\sum_{k=1}^{m-1}\frac{\pi^{2k}}{(2k+1)!}\zeta(2m-2k)=0.$$
Recently, Merca~\cite{Me0} obtained an alternative proof for the above equality using the generating function of Bernoulli numbers, and ~\cite{Me2} also obtained a homogeneous linear recurrence relation for $\zeta(2m)$ using several tools from symmetric function theory. Extending the above results, in 2017 Merca \cite{Me} also introduced a number of infinite families of linear recurrence relations between $\zeta(2m).$

In 2013, Lettington (\cite{Le}, \cite[(1.24)]{Le1}) proved the following linear recurrence relation between the special values $\lambda(2m)$ of the Dirichlet lambda function $\lambda(s)$.

\begin{theorem}[Lettington]\label{thm-main}
Let $m$ be a positive integer. Then
$$\lambda(2m)=(-1)^{m-1}\left(\frac{\pi^{2m}}{4(2m)!}+\sum_{k=1}^{m-1}\frac{(-1)^{m-k}\pi^{2k}}{(2k+1)!}\lambda(2m-2k) \right).$$
\end{theorem}

 In this study, based on the integral representation of
\begin{equation}\label{Dirichlet eta}
J(s,a)=\sum_{n=0}^\infty\frac{(-1)^n}{(n+a)^{s}},\quad\text{Re}(s) > 0,
\end{equation}
(see (\ref{J-int-re})) and an analogue of the Hermite formula for $J(s,a)$ (see (\ref{J-int-re-2})), we provide a new proof of Theorem \ref{thm-main}.

Analogous to Merca's work~\cite{Me},
we further prove a number of infinite families of linear recurrence relations for $\lambda(2m)$ using the generating function of the Euler polynomials and the Euler-type formula (\ref{lam-ft-eq}). (See Theorems \ref{thm1}, \ref{thm2}, \ref{thm4}, \ref{thm5}, and \ref{thm6} and Corollaries \ref{coro1}, \ref{coro2}, and \ref{cor3}).

Moreover, we prove convolution identities for the special values of $\lambda(s)$ at even arguments and for special values
of $\zeta(s)$ and $\beta(s)$ at odd arguments. (See Theorems \ref{Conv2} and \ref{Conv3}).

Finally, in Proposition \ref{pro-main} below, we demonstrate a power series expansion for the function $J(s,a)$ (\ref{Dirichlet eta}), which implies a known
one from Coffey for the Dirichlet eta functions
$$\eta(s)=\sum_{n=1}^\infty\frac{(-1)^{n-1}}{n^s},\quad\text{Re}(s) > 0,$$
 (see Corollary \ref{Coffey+}).

In the following, we shall employ the usual convention that an empty sum is taken to be zero. For example, if $m=1,$ then we understand that $\sum_{k=1}^{m-1}=0$.

\subsection{Linear recurrence relations for the lambda function at even integers}

\begin{theorem}\label{thm1}
Let $m$ be a positive integer, and $\alpha$ a complex number such that $\alpha\neq\frac12.$ Then we have
$$\begin{aligned}(-1)^{m+1}&\frac{\alpha^{2m}-(\alpha-1)^{2m}}{4}\cdot\frac{\pi^{2m+2}}{(2m)!} \\
&+\sum_{k=1}^m(-1)^k(\alpha^{2k-1}+(\alpha-1)^{2k-1})\frac{\pi^{2k}}{(2k-1)!}\lambda(2m-2k+2)=0.
\end{aligned}$$
 \end{theorem}

\begin{theorem}\label{thm2}
Let $m$ be a nonnegative integer, and $\alpha$ a complex number. Then we have
$$\begin{aligned}(-1)^{m+1}&\frac{\alpha^{2m+1}-(\alpha-1)^{2m+1}}{4}\cdot\frac{\pi^{2m+2}}{(2m+1)!} \\
&+\sum_{k=0}^m(-1)^k(\alpha^{2k}+(\alpha-1)^{2k})\frac{\pi^{2k}}{(2k)!}\lambda(2m-2k+2)=0,
\end{aligned}$$
where define $0^0=1.$
 \end{theorem}

Let us state some special cases of Theorem \ref{thm2}.

\begin{corollary}[$\alpha=1$]\label{coro1}
Let $m$ be a nonnegative integer. Then we have
$$\begin{aligned}(-1)^{m+1}\frac{\pi^{2m+2}}{4(2m+1)!}
+\sum_{k=0}^m(-1)^k(1+\delta_k)\frac{\pi^{2k}}{(2k)!}\lambda(2m-2k+2)=0,
\end{aligned}$$
where $\delta_k=0$ if $k\neq0$ and 1 if $k=0.$
 \end{corollary}

\begin{corollary}[$\alpha=\frac12$]\label{coro2}
Let $m$ be a nonnegative integer. Then we have
$$\begin{aligned}(-1)^{m+1}\frac{\pi^{2m+2}}{4(2m+1)!}
+\sum_{k=0}^m(-1)^k\frac{2^{2m-2k+1}\pi^{2k}}{(2k)!}\lambda(2m-2k+2)=0.
\end{aligned}$$
\end{corollary}

\begin{remark}
If $m$ is a positive integer,
then from Corollary \ref{coro2} we have the following result:
$$\lambda(2m)=(-1)^{m-1}\left(\frac{\pi^{2m}}{2^{2m+1}(2m-1)!}+\sum_{k=1}^{m-1}(-1)^{m-k}\frac{\pi^{2k}}{2^{2k}(2k)!}\lambda(2m-2k)\right),$$
which is an analogue of a result by Lettington (see Theorem \ref{thm-main}).
We observe that this linear recurrence relation does not require a priori knowledge on the Euler numbers.
If $m$ is small, then it is easy to evaluate $\lambda(2m)$ as follows:
$$\lambda(2)=\frac{\pi^2}8,~~ \lambda(4)=\frac{\pi^4}{96},~~\lambda(6)=\frac{\pi^6}{960},~~\lambda(8)=\frac{17\pi^8}{161280}.$$
\end{remark}

\begin{corollary}\label{cor3}
Let $n$ be a positive integer.
\begin{enumerate}
\item[$(1)$] For integers $m>0,$ we have
$$\begin{aligned}
\frac{(-1)^{m+1}}{4}&\left(2\sum_{j=1}^{n-1}(-1)^jj^{2m}+(-1)^n n^{2m}\right)\frac{\pi^{2m+2}}{(2m)!} \\
&+(-1)^n\sum_{k=1}^m(-1)^k n^{2k-1}\frac{\pi^{2k}}{(2k-1)!}\lambda(2m-2k+2)=0.
\end{aligned}$$
\item[$(2)$] For integers $m\geq0,$ we have
$$\begin{aligned}
\frac{(-1)^{m+1}}{4}&\left(2\sum_{j=1}^{n-1}(-1)^jj^{2m+1}+(-1)^n n^{2m+1}\right)\frac{\pi^{2m+2}}{(2m+1)!} \\
&+\sum_{k=0}^m(-1)^k ((-1)^nn^{2k}-\delta_k)\frac{\pi^{2k}}{(2k)!}\lambda(2m-2k+2)=0.
\end{aligned}$$
\end{enumerate}
 \end{corollary}

\begin{theorem}\label{thm4}
Let $m$ be a positive integer, and $\alpha$ a complex number. Then we have
$$\begin{aligned}(-1)^{m}&\frac{\pi^{2m+2}}{4(2m)!} ((\alpha-1)^{2m}+(\alpha+1)^{2m}-2\alpha^{2m}) \\
&+\sum_{k=1}^m(-1)^k((\alpha-1)^{2k-1}-(\alpha+1)^{2k-1})\frac{\pi^{2k}}{(2k-1)!}\lambda(2m-2k+2)=0.
\end{aligned}$$
 \end{theorem}

\begin{theorem}\label{thm5}
Let $m$ be a nonnegative integer, and $\alpha$ a complex number. Then we have
$$\begin{aligned}(-1)^{m}&\frac{\pi^{2m+2}}{4(2m+1)!}((\alpha-1)^{2m+1}-(\alpha+1)^{2m+1}) \\
&+\sum_{k=0}^m(-1)^{k}((\alpha-1)^{2k}+(\alpha+1)^{2k}+2\alpha^{2k})\frac{\pi^{2k}}{(2k)!}\lambda(2m-2k+2)=0,
\end{aligned}$$
where define $0^0=1.$
 \end{theorem}

\begin{theorem}\label{thm6}
Let $m$ be a nonnegative integer. Then we have
$$\begin{aligned}
(-1)^{m+1}&2(1-3^{-2m-1})\lambda(2m+2) \\
&=\frac{\pi^{2m+2}}{3^{2m+1}(2m+1)!}+4\sum_{k=0}^m(-1)^{k+1}\frac{\pi^{2m-2k}}{3^{2m-2k}(2m-2k)!}\lambda(2k+2).
\end{aligned}$$
\end{theorem}

\section{An alternative proof of Theorem \ref{thm-main}}

For our purpose, we require the following lemmas.

\begin{lemma}\label{lem1-1}
Let $i=\sqrt {-1}.$ Then
$$\begin{aligned}
(1+ix)^a-(1-ix)^a={2i}(1+x^2)^{\frac a2}\sin\left(a\tan^{-1}x\right).
\end{aligned}$$
\end{lemma}
\begin{proof}
We have the following identity with $y=a\tan^{-1}x$:
$$\sin y=\frac{1}{2i}\left[\left(\frac{1+ix}{1-ix}\right)^{\frac a2}-\left(\frac{1+ix}{1-ix}\right)^{-\frac a2} \right].$$
From this, we obtain
$$\sin\left(a\tan^{-1}x\right)=\frac1{{2i}}\left[\frac{(1+ix)^a}{(1+x^2)^{\frac a2}}-\frac{(1-ix)^a}{(1+x^2)^{\frac a2}}\right],$$
which is the desired result.
\end{proof}

\begin{remark}
Here, we note that a similar expression to Lemma \ref{lem1-1} was given by Adamchik (see, e.g., {\cite[p. 4, (9)]{Ad}}) without proof.
\end{remark}

\begin{lemma}\label{lem2-1}
Let $m$ be a positive integer and $i=\sqrt {-1}.$ Then we have
$$\begin{aligned}
\sum_{k=1}^{m}(-1)^k\binom{2m}{2k-1}\left(\frac t\pi\right)^{2k-1}
=\frac{i}{2\pi^{2m}}\left[(\pi+it)^{2m}-(\pi-it)^{2m}\right].
\end{aligned}$$
\end{lemma}
\begin{proof}
Applying the binomial expansion, we have
$$\begin{aligned}
i\left[(\pi+it)^{2m}-(\pi-it)^{2m}\right]&=i\pi^{2m}\sum_{k=0}^{2m}(1-(-1)^k)\binom{2m}{k}\left(\frac{it}\pi\right)^k \\
&=i\pi^{2m}\sum_{k=1}^{m}2\binom{2m}{2k-1}\left(\frac{it}\pi\right)^{2k-1} \\
&=2\pi^{2m}\sum_{k=1}^{m}(-1)^k\binom{2m}{2k-1}\left(\frac{t}\pi\right)^{2k-1},
\end{aligned}$$
which completes the proof.
\end{proof}

\begin{lemma}\label{lem3-1}
Let $m$ be a positive integer. Then we have
$$\begin{aligned}
\sum_{k=1}^{m}(-1)^k \frac{\pi^{2m-2k}}{(2m-2k+1)!}\frac{t^{2k-1}}{(2k-1)!}
=-\frac{(\pi^2+t^2)^m}{\pi(2m)!}\sin\left(2m\tan^{-1}\left(\frac t\pi\right)\right).
\end{aligned}$$
\end{lemma}
\begin{proof}
Applying the binomial expansion, we have
$$\begin{aligned}
\sum_{k=1}^{m}(-1)^k& \frac{\pi^{2m-2k}}{(2m-2k+1)!}\frac{t^{2k-1}}{(2k-1)!} \\
&=\sum_{k=1}^{m}(-1)^k \frac{(2m)!\pi^{-2k+1}t^{2k-1}}{(2m-2k+1)!(2k-1)!}\frac{\pi^{2m-1}}{(2m)!} \\
&=\frac{\pi^{2m-1}}{(2m)!}\sum_{k=1}^{m}(-1)^k \binom{2m}{2k-1}\left(\frac t\pi\right)^{2k-1} \\
&=\frac{i}{2\pi(2m)!}\left[(\pi+it)^{2m}-(\pi-it)^{2m}\right] \\
&\quad\text{(by Lemma \ref{lem2-1})} \\
&=-\frac{1}{\pi(2m)!}(\pi^2+t^2)^m \sin\left(2m\tan^{-1}\left(\frac t\pi\right)\right) \\
&\quad\text{(by Lemma \ref{lem1-1} with $a=2m$ and $x=t/\pi$}). \\
\end{aligned}$$
This completes the proof.
\end{proof}

\begin{lemma}\label{lem4-1}
Let $m$ be a positive integer. Then we have
$$\int_0^\infty(\pi^2+t^2)^m \sin\left(2m\tan^{-1}\left(\frac t\pi\right)\right)\frac{e^t dt}{e^{2t}-1}=\frac{\pi^{2m+1}}{4}.$$
\end{lemma}
\begin{proof}
Recall that the function $J(s,a)$ is defined as follows:
\begin{equation}\label{J-def}
J(s,a)=\sum_{n=0}^\infty\frac{(-1)^n}{(n+a)^{s}},
\end{equation}
where $0<a\leq 1$ (see \cite[(1.1)]{WY}).
Using the classical Hurwitz zeta functions
\begin{equation}\label{H-def}
\zeta(s,a)=\sum_{n=0}^\infty\frac{1}{(n+a)^{s}},
\end{equation}
we also have the integral representation of $J(s,a)$
\begin{equation}\label{J-int-re}
\Gamma(s)J(s,a)=\int_{0}^\infty\frac{e^{(1-a)t}x^{s-1}}{e^t+1}dt,\quad\text{Re}(s)>0,
\end{equation}  (see \cite[(3.1)]{WY})
and an analogue of Hermite's formula for $J(s,a),$
\begin{equation}\label{J-int-re-2}
J(s,a)=\frac{a^{-s}}{2}+2\int_0^\infty(a^2+y^2)^{-s/2}\sin\left(s\tan^{-1}\frac ya\right)\frac{e^{\pi y}dy}{e^{2\pi y}-1}.
\end{equation}
The above formula enables $J(s,a)$ to be analytically continued to the whole complex plane.
The special values of $J(s,a)$ at non-positive integers can be represented by Euler polynomials as follows:
\begin{equation}\label{J-values}
J(-m,a)=\frac{(-1)^m}{2}E_m(1-a)=\frac12 E_m(a), \quad m\geq0,
\end{equation} (see  \cite[(3.8)]{WY}). Letting $t=\pi y, a=1,$ and $s=-2m$ in (\ref{J-int-re-2}), we have
\begin{equation}\label{J-int-re-3}
\begin{aligned}
J(-2m,1)-\frac{1}{2}&=-\frac2{\pi}\int_0^\infty\left(1+\left(\frac t\pi\right)^2\right)^{m}\sin\left(2m\tan^{-1}\left(\frac t\pi\right)\right)
\frac{e^{t}dt}{e^{2t}-1} \\
&=-\frac{2}{\pi^{2m+1}}\int_0^\infty\left(\pi^2+t^2\right)^{m}\sin\left(2m\tan^{-1}\left(\frac t\pi\right)\right)
\frac{e^{t}dt}{e^{2t}-1}.
\end{aligned}
\end{equation}
Replacing $-m$ by $-2m$ and setting $a=1$ in (\ref{J-values}), we have
\begin{equation}\label{J-values-2}
J(-2m,1)=\frac{(-1)^{-2m}}{2}E_{2m}(0)=\frac12 E_{2m}(1)=0,\quad m\geq1,
\end{equation}
because $E_{2m}(0)=0$ for $m>1.$ Combining (\ref{J-int-re-3}) and (\ref{J-values-2}), we obtain the desired result.
\end{proof}

\begin{proof}[Proof of Theorem \ref{thm-main}]
The analytic continuation of $\lambda(s)$ can be established using Riemann's method. According to Euler,
we have the following integral representation of the Gamma function:
\begin{equation}\label{gamm-1}
\frac{\Gamma(s)}{(2n+1)^s}=\int_0^\infty e^{-(2n+1)t}t^{s-1}dt, \quad\text{Re}(s)>1.
\end{equation}
Summing both sides over $n,$ we have
\begin{equation}\label{gamm-2}
\begin{aligned}
\Gamma(s)\sum_{n=0}^\infty\frac{1}{(2n+1)^s}&=\sum_{n=0}^\infty\int_0^\infty e^{-(2n+1)t}t^{s-1}dt \\
&=\int_0^\infty\sum_{n=0}^\infty e^{-(2n+1)t}t^{s-1}dt \\
&=\int_0^\infty\frac{e^{-t}t^{s-1}}{1-e^{-2t}}dt.
\end{aligned}
\end{equation}
That is, for Re$(s)>1,$
\begin{equation}\label{gamm-3}
\lambda(s)=\frac1{\Gamma(s)}\int_0^\infty\frac{e^{t}t^{s-1}}{e^{2t}-1}dt, \quad\text{Re}(s)>1.
\end{equation}
Then, applying the integral representation   (\ref{gamm-3}), we have
\begin{equation}\label{main-1}
\begin{aligned}
\sum_{k=1}^{m-1}&\frac{(-1)^{k}\pi^{2m-2k}}{(2m-2k+1)!}\lambda(2k) \\
&=\sum_{k=1}^{m-1}\frac{(-1)^{k}\pi^{2m-2k}}{(2m-2k+1)!} \frac1{\Gamma(2k)}\int_0^\infty\frac{e^{t}t^{2k-1}}{e^{2t}-1}dt \\
&=\int_{0}^\infty\left[\sum_{k=1}^{m-1}\frac{(-1)^{k}\pi^{2m-2k}}{(2m-2k+1)!} \frac{t^{2k-1}}{(2k-1)!}\right]\frac{e^{t}}{e^{2t}-1}dt \\
&=\int_{0}^\infty\left[\sum_{k=1}^{m}\frac{(-1)^{k}\pi^{2m-2k}}{(2m-2k+1)!} \frac{t^{2k-1}}{(2k-1)!}\right]\frac{e^{t}}{e^{2t}-1}dt \\
&\quad+(-1)^{m+1}\frac{1}{(2m-1)!}\int_{0}^\infty\frac{e^{t}t^{2m-1}}{e^{2t}-1}dt,
\end{aligned}
\end{equation}
and applying Lemma \ref{lem3-1} and (\ref{gamm-3}) to the last line of (\ref{main-1}) with $s=2m$, we derive the following identity:
\begin{equation}\label{main-2}
\begin{aligned}
\sum_{k=1}^{m-1}&\frac{(-1)^{k}\pi^{2m-2k}}{(2m-2k+1)!}\lambda(2k) \\
&=-\frac{1}{\pi(2m)!}\int_{0}^\infty\left[(\pi^2+t^2)^m\sin\left(2m\tan^{-1}\left(\frac t\pi\right)\right)\right]\frac{e^{t}}{e^{2t}-1}dt \\
&\quad+(-1)^{m+1}\lambda(2m).
\end{aligned}
\end{equation}
Finally, applying Lemma \ref{lem4-1} to the right-hand side of the above equality, we obtain
$$\sum_{k=1}^{m-1}\frac{(-1)^{k}\pi^{2m-2k}}{(2m-2k+1)!}\lambda(2k)=-\frac{\pi^{2m}}{4(2m)!}+(-1)^{m+1}\lambda(2m),$$
and the desired result follows from a direct manipulation.
\end{proof}

\section{Proofs of Theorems \ref{thm1} and \ref{thm2} and Corollary \ref{cor3}}

First, we prove the following lemma:

\begin{lemma}\label{lem1}
Let $n$ be a nonnegative integer, and $x$ and $\alpha$ complex numbers. Then
$$\begin{aligned}
\sum_{k=0}^n\left(\alpha^{n-k}-(-1)^k(2x-1+\alpha)^{n-k}\right)\binom nkE_k(x)=0.
\end{aligned}$$
\end{lemma}

\begin{remark}
Merca proved a similar result for the Bernoulli polynomials (see \cite[Theorem 2.1]{Me}).
\end{remark}

\begin{remark}
Letting $\alpha=0$ in Lemma \ref{lem1}
and recalling that $0^0=1,$ we have
$$(1-(-1)^n)E_n(x)-\sum_{k=0}^{n-1}(-1)^k(2x-1)^{n-k}\binom nkE_k(x)=0$$
for $n\geq1,$
which can be written as
$$E_n(x)-\sum_{k=0}^{n}(-1)^k(2x-1)^{n-k}\binom nkE_k(x)=0.$$
Then, by (\ref{lam-ft-eq}) we have the following recurrence relation for $\lambda(s)$ at positive even integers:
$$\lambda(2m+2)+(-1)^{m+1}\pi^{2m+2}\left(\frac1{4(2m+1)!}+\sum_{k=0}^m(-1)^{k+1}\frac{\lambda(2k+2)}{\pi^{2k+2}(2m-2k)!}\right)=0.$$
\end{remark}

\begin{proof}[Proof of Lemma \ref{lem1}.]
We consider (\ref{euler poly-def}) and the series
$\sum_{n=0}^\infty\alpha^n\frac{t^n}{n!}=e^{\alpha t}.$
For $|t|<\pi,$ it holds that
\begin{equation}\label{lem12}
\begin{aligned}
\frac{2e^{tx}}{e^t+1}e^{\alpha t}=\left(\sum_{n=0}^\infty E_n(x)\frac{t^n}{n!}\right)&\left(\sum_{n=0}^\infty\alpha^n\frac{t^n}{n!}\right),
\end{aligned}
\end{equation}
which can be written as
\begin{equation}\label{lem11}
\begin{aligned}
\frac{2e^{tx}}{e^t+1}e^{\alpha t}
&=\frac{2e^{x(-t)}}{1+e^{-t}}e^{(2x+\alpha-1)t} \\
&=\left(\sum_{n=0}^\infty(-1)^n E_n(x)\frac{t^n}{n!}\right)\left(\sum_{n=0}^\infty(2x+\alpha-1)^n\frac{t^n}{n!}\right).
\end{aligned}
\end{equation}
By applying Cauchy's rule for the multiplication of two power series to the right-hand sides of (\ref{lem12}) and (\ref{lem11}) and then comparing the coefficients of $t^{n}$, we obtain the relation
$$\sum_{k=0}^n\binom nk E_k(x)\alpha^{n-k}=\sum_{k=0}^n(-1)^k\binom nkE_k(x)(2x+\alpha-1)^{n-k}$$
for $n\geq0,$ and the desired result follows from a direct manipulation.
\end{proof}

\begin{proof}[Proof of Theorem \ref{thm1}]
Setting $x=0$ and $n=2m$ in Lemma \ref{lem1}, we have
$$\sum_{k=0}^{2m}\left(\alpha^{2m-k}-(-1)^k(\alpha-1)^{2m-k}\right)\binom{2m}k E_k(0)=0$$
for $m\geq0,$ which can be written as
$$\sum_{k=1}^m\frac{\alpha^{2k-1}+(\alpha-1)^{2k-1}}{(2k-1)!(2m-2k+1)!}E_{2m-2k+1}(0)=-\frac{\alpha^{2m}-(\alpha-1)^{2m}}{(2m)!}$$
for $m>0,$ because $E_0(0)=1$ and $E_{2m}(0)=0$ for $m>1.$
Then, using Euler's formula (\ref{lam-ft-eq}) for the Dirichlet lambda functions, we obtain the relation
$$\begin{aligned}
\sum_{k=1}^m\frac{\alpha^{2k-1}+(\alpha-1)^{2k-1}}{(2k-1)!}&(-1)^{m-k+1}\frac{4}{\pi^{2m-2k+2}}\lambda(2m-2k+2) \\
&=-\frac{\alpha^{2m}-(\alpha-1)^{2m}}{(2m)!},
\end{aligned}$$
where $\alpha\neq\frac12,$
and a direct manipulation implies the desired result.
\end{proof}

\begin{proof}[Proof of Theorem \ref{thm2}]
Letting $n=2m+1$ and $x=0$ in Lemma \ref{lem1}, we have
$$\sum_{k=0}^{2m+1}\left(\alpha^{2m+1-k}-(-1)^k(\alpha-1)^{2m+1-k}\right)\binom{2m+1}k E_k(0)=0$$
for $m\geq0,$ which can be written as
$$\sum_{k=0}^m\frac{\alpha^{2k}+(\alpha-1)^{2k}}{(2k)!(2m-2k+1)!}E_{2m-2k+1}(0)=-\frac{\alpha^{2m+1}-(\alpha-1)^{2m+1}}{(2m+1)!}$$
for $m\geq0,$ because $E_0(0)=1$ and $E_{2m}(0)=0$ for $m>1.$
Then, using Euler's formula (\ref{lam-ft-eq}) for the Dirichlet lambda functions, we obtain the relation
$$\begin{aligned}
\sum_{k=0}^m\frac{\alpha^{2k}+(\alpha-1)^{2k}}{(2k)!}&(-1)^{m-k+1}\frac{4}{\pi^{2m-2k+2}}\lambda(2m-2k+2) \\
&=-\frac{\alpha^{2m+1}-(\alpha-1)^{2m+1}}{(2m+1)!},
\end{aligned}$$
where $m\geq0$ (define $0^0=1$),
and the desired result follows from a direct manipulation.
\end{proof}

\begin{proof}[Proof of Corollary \ref{cor3}]
For $n>0,$ it is easy to observe that
\begin{equation}\label{1+}
\sum_{j=1}^n(-1)^j(j^{2m}-(j-1)^{2m})=2\sum_{j=1}^{n-1}(-1)^jj^{2m}+(-1)^n n^{2m}
\end{equation}
and
\begin{equation}\label{2+}
\sum_{j=1}^n(-1)^j(j^{2k-1}+(j-1)^{2k-1})=(-1)^n n^{2k-1}.
\end{equation}
Letting $\alpha=j~(j\geq1)$ in Theorem \ref{thm1}, we have
\begin{equation}\label{3+}
\begin{aligned}
\sum_{j=1}^n(-1)^j &\biggl((-1)^{m+1}\frac{j^{2m}-(j-1)^{2m}}{4}\cdot\frac{\pi^{2m+2}}{(2m)!} \\
&+\sum_{k=1}^m(-1)^k(j^{2k-1}+(j-1)^{2k-1})\frac{\pi^{2k}}{(2k-1)!}\lambda(2m-2k+2)\biggl)=0.
\end{aligned}
\end{equation}
Then, substituting (\ref{1+}) and (\ref{2+}) into (\ref{3+}), we obtain
\begin{equation}\label{4+}
\begin{aligned}
\frac{(-1)^{m+1}}{4}&\left(2\sum_{j=1}^{n-1}(-1)^jj^{2m}+(-1)^n n^{2m}\right)\frac{\pi^{2m+2}}{(2m)!} \\
&+(-1)^n\sum_{k=1}^m(-1)^kn^{2k-1}\frac{\pi^{2k}}{(2k-1)!}\lambda(2m-2k+2)=0,
\end{aligned}
\end{equation}
which implies the first part.
The second part follows from the same reasoning.
\end{proof}

\section{Proofs of Theorems \ref{thm4}, \ref{thm5}, and \ref{thm6}}
In this section, we employ the generating functions of the Euler polynomials (\ref{euler poly-def}) to derive another three families of infinite recurrence relations
for the Dirichlet lambda functions at positive even integer arguments.

\begin{lemma}\label{lem2}
Let $n$ be a nonnegative integer, and $x$ and $\alpha$ complex numbers. Then we have
$$\begin{aligned}
\sum_{k=0}^n(-1)^{n-k}&\binom nk\frac{(2x+\alpha-1)^k+(2x-\alpha-1)^k}2E_{n-k}(x)=\sum_{k=0}^{\left\lfloor\frac n2 \right\rfloor}\binom{n}{2k}E_{n-2k}(x)\alpha^{2k},
\end{aligned}$$
where $\lfloor\cdot\rfloor$ is the floor function.
\end{lemma}

\begin{remark}
Merca proved a similar result for the Bernoulli polynomials (see \cite[Theorem 3.1]{Me}).
\end{remark}

\begin{proof}[Proof of Lemma \ref{lem2}.]
Using the generating function of Euler polynomials $E_n(x)$  (\ref{euler poly-def}) and the power series expansion of $e^{t},$ we obtain
\begin{equation}\label{eq-1}
\begin{aligned}
\biggl(\sum_{n=0}^\infty &E_n(x)\frac{t^n}{n!}\biggl)\biggl(\sum_{n=0}^\infty\alpha^{2n}\frac{t^{2n}}{(2n)!}\biggl) \\
&=\frac12\left(\sum_{n=0}^\infty E_n(x)\frac{t^n}{n!}\right)
\left(\sum_{n=0}^\infty\alpha^{n}\frac{t^{n}}{n!}+\sum_{n=0}^\infty(-1)^n\alpha^{n}\frac{t^{n}}{n!}\right) \\
&=\frac12\frac{2e^{xt}}{e^t+1}(e^{\alpha t}+e^{-\alpha t}) \\
&=\frac12\frac{2e^{x(-t)}}{e^{-t}+1}(e^{(2x+\alpha-1)t}+e^{(2x-\alpha-1)t}) \\
&=\frac12\left(\sum_{n=0}^\infty (-1)^nE_n(x)\frac{t^n}{n!}\right)\left(\sum_{n=0}^\infty((2x+\alpha-1)^n+(2x-\alpha-1)^n)\frac{t^n}{n!}\right) \\
&=\frac12\sum_{n=0}^\infty\left(\sum_{k=0}^n(-1)^{n-k}
\frac{(2x+\alpha-1)^k+(2x-\alpha-1)^k}{(n-k)!k!}E_{n-k}(x)\right){t^{n}},
\end{aligned}
\end{equation}
where the final equality is obtained by Cauchy's rule for multiplying power series.
By again applying Cauchy's rule, we derive the following identity:
\begin{equation}\label{eq-2}
\begin{aligned}
\left(\sum_{n=0}^\infty E_n(x)\frac{t^n}{n!}\right)\left(\sum_{n=0}^\infty\alpha^{2n}\frac{t^{2n}}{(2n)!}\right)
 =\sum_{n=0}^\infty\left( \sum_{k=0}^{\left\lfloor\frac n2 \right\rfloor}\frac{\alpha^{2k}}{(n-2k)!(2k)!}E_{n-2k}(x) \right){t^{n}},
\end{aligned}
\end{equation}
where $\lfloor\cdot\rfloor$ is the floor function.
Therefore, comparing the coefficients of $t^n$ in (\ref{eq-1}) and (\ref{eq-2}), we obtain the desired result.
\end{proof}

\begin{proof}[Proof of Theorem \ref{thm4}] Setting $x=0$ and $n=2m$ in Lemma \ref{lem2}, we have
$$\alpha^{2m}=\frac12\sum_{k=0}^{2m}(-1)^{2m-k}\binom{2m}k\left((\alpha-1)^k+(-\alpha-1)^k \right)E_{2m-k}(0)$$
for $m\geq0.$ Because $E_{2m}(0)=0~ (m>1)$ and $E_0(0)=1,$ after moving the $2m$-th term to the left-hand side we obtain
$$\begin{aligned}
\alpha^{2m}&-\frac12\left((\alpha-1)^{2m}+(\alpha+1)^{2m}\right) \\
&=-\frac12\sum_{k=1}^{m}\binom{2m}{2k-1}\left((\alpha-1)^{2k-1}-(\alpha+1)^{2k-1} \right)E_{2m-(2k-1)}(0)
\end{aligned}$$
for $m\geq1.$ Finally, by applying Euler's formula (\ref{lam-ft-eq}) for the Dirichlet lambda functions to the above equality, we obtain the desired result.
\end{proof}

\begin{proof}[Proof of Theorem \ref{thm5}]
Setting $x=0$ and $n=2m+1$ in Lemma \ref{lem2}, we have
$$\begin{aligned}
\sum_{k=0}^{\left\lfloor\frac{2m+1}2\right\rfloor}&\binom{2m+1}{2k}E_{2m-2k+1}(0)\alpha^{2k} \\
=\frac12\sum_{k=0}^{2m+1}(-1)^{2m-k+1}&\binom {2m+1}k\left((\alpha-1)^k+(-\alpha-1)^k\right)E_{2m-k+1}(0)
\end{aligned}$$
for $m\geq0.$ Because $E_{2m}(0)=0~(m>1)$ and $E_0(0)=1,$ after moving the $(2m+1)$-th term to the left-hand side we obtain
$$\begin{aligned}
\sum_{k=0}^{m}&\binom{2m+1}{2k}E_{2m-2k+1}(0)\alpha^{2k}-\frac12\left((\alpha-1)^{2m+1}-(\alpha+1)^{2m+1}\right) \\
&=\frac12\sum_{k=0}^{m}(-1)^{2m-2k+1}\binom {2m+1}{2k}\left((\alpha-1)^{2k}+(-\alpha-1)^{2k}\right)E_{2m-2k+1}(0),
\end{aligned}$$
or equivalently
$$\begin{aligned}
\sum_{k=0}^{m}&\binom{2m+1}{2k}\left(\frac12\left((\alpha-1)^{2k}+(\alpha+1)^{2k}\right)+\alpha^{2k}\right) E_{2m-2k+1}(0) \\
&=\frac12\left((\alpha-1)^{2m+1}-(\alpha+1)^{2m+1}\right)
\end{aligned}$$
for $m\geq0.$
Finally, by applying Euler's formula (\ref{lam-ft-eq}) to the above equality, we obtain the relation
$$\begin{aligned}
\sum_{k=0}^{m}(-1)^{k+1}&\frac{(2m+1)!}{(2k)!}\left(\frac12\left((\alpha-1)^{2k}+(\alpha+1)^{2k}\right)+\alpha^{2k}\right)\pi^{2k}\lambda(2m-2k+2) \\
&=\frac18(-1)^m\pi^{2m+2}\left((\alpha-1)^{2m+1}-(\alpha+1)^{2m+1}\right)
\end{aligned}$$
for $m\geq0,$ which is the desired  result.
\end{proof}

\begin{proof}[Proof of Theorem \ref{thm6}]
From the expression for the Euler polynomials
$$E_n(x)=\sum_{k=0}^n\binom nkx^{n-k}E_k(0),$$
we have
\begin{equation}\label{eq-mu}
\begin{aligned}
E_{2m+1}\left(\frac13\right)&=\sum_{k=0}^{2m+1}\binom {2m+1}kE_k(0)\left(\frac13\right)^{2m+1-k} \\
&=\left(\frac13\right)^{2m+1}+\sum_{k=0}^{m}\binom {2m+1}{2k+1}E_{2k+1}(0)\left(\frac13\right)^{2m-2k}
\end{aligned}
\end{equation}
because $E_{2m}(0)=0~ (m>1)$ and $E_0(0)=1.$ It is known from \cite[Theorem 3.3]{sun} that
\begin{equation}\label{sun-eq1}
E_n(x)=m^n\sum_{k=0}^{m-1}(-1)^kE_n\left(\frac{x+k}{m}\right)\quad\text{if }2\nmid m.
\end{equation}
Setting $x=0$ and $m=3$ in (\ref{sun-eq1}), and noticing that $E_{2m+1}\left(\frac23\right)=-E_{2m+1}\left(\frac13\right),$ we have
\begin{equation}\label{eq1/3}
E_{2m+1}\left(\frac13\right)=\frac12(1-3^{-2m-1})E_{2m+1}(0).
\end{equation}
Comparing (\ref{eq-mu}) and (\ref{eq1/3}), we find that
\begin{equation}\label{con-eq}
\frac12(1-3^{-2m-1})E_{2m+1}(0)=
\left(\frac13\right)^{2m+1}+\sum_{k=0}^{m}\binom {2m+1}{2k+1}E_{2k+1}(0)\left(\frac13\right)^{2m-2k}.
\end{equation}
Then, by applying Euler's formula (\ref{lam-ft-eq}) for the Dirichlet lambda functions to the above equality, we obtain the desired result.
\end{proof}

\section{Convolution identities}

In this section, we prove convolution identities for special values of $\lambda(s)$ at even arguments and for special values
of $\beta(s)$ at odd arguments.

Euler derived the following beautiful convolution identity for the Bernoulli numbers:
\begin{equation}\label{Ber-id}
\sum_{k=1}^{m-1}\binom{2m}{2k}B_{2k}B_{2m-2k}=-(2m+1)B_{2m}, \quad m\geq2,
\end{equation}
which has been generalized by many authors in different directions (see, for example, \cite{Co1}, \cite{Di}, and \cite{KHR}).
By Euler's formula for the special values of the Riemann zeta function at even arguments,
$$\zeta(2m)=(-1)^{m-1}\frac{(2\pi)^{2m}B_{2m}}{2(2m)!},$$
 (\ref{Ber-id}) can be written as
$$\sum_{k=1}^{m-1}\zeta(2k)\zeta(2m-2k)=\left(m+\frac12\right)\zeta(2m), \quad m\geq2.$$

Similarly, the following convolution identity for the special values of $\lambda(s)$ at even arguments has been demonstrated in \cite{Di,SD} for $m\geq1$:
\begin{equation}\label{Eu-con-id}
\sum_{k=1}^m\lambda(2k)\lambda(2m-2k+2)=\left(m+\frac12\right)\lambda(2m+2).
\end{equation}

In addition, the following convolution identity for the special values of  $\beta(s)$ at odd arguments was proved by Williams (see \cite[p.~22, Theorem II]{WG}). Here, we present an alternative proof.

\begin{theorem}[G.T. Williams]\label{Conv2}
Let $m$ be a nonnegative integer. Then we have
$$\sum_{k=0}^m\beta(2k+1)\beta(2m-2k+1)=\left(m+\frac12\right)\lambda(2m+2).$$
\end{theorem}
\begin{proof}
Applying the generating function of the Euler polynomials, we have
$$\begin{aligned}
\frac{2^2e^t}{(e^t+1)^2}&=\left(\sum_{n=0}^\infty E_n\left(\frac12\right)\frac{t^n}{n!}\right)
\left(\sum_{m=0}^\infty E_m\left(\frac12\right)\frac{t^m}{m!}\right) \\
&=\sum_{n=0}^\infty\sum_{k=0}^nE_k\left(\frac12\right)E_{n-k}\left(\frac12\right)\frac{t^n}{k!(n-k)!} \\
&=\sum_{n=0}^\infty\sum_{k=0}^n\binom nkE_k\left(\frac12\right)E_{n-k}\left(\frac12\right)\frac{t^n}{n!}.
\end{aligned}$$
We also have
$$\begin{aligned}
\frac12\left(\frac{2e^{\frac12 t}}{e^t+1}\right)^2
&=\frac{d}{dt}\left(\frac{2e^t}{e^t+1}\right)=\frac{d}{dt}\left(\sum_{n=0}^\infty E_n(1)\frac{t^n}{n!}\right) \\
&=\sum_{n=0}^\infty E_{n+1}(1)\frac{t^n}{n!}.
\end{aligned}$$
Comparing the coefficients of $t^n$ in the above two equalities, we have
\begin{equation}\label{2-Euler-eq}
\sum_{k=0}^n\frac{E_k\left(\frac12\right)E_{n-k}\left(\frac12\right)}{k!(n-k)!}=2\frac{E_{n+1}(1)}{n!}.
\end{equation}
Then, setting $n=2m$ in (\ref{2-Euler-eq}) and noticing that $E_{2m+1}(1)=(-1)^{2m+1}E_{2m+1}(0)$ and $E_{k}=2^kE_k(1/2),$ we obtain
\begin{equation}\label{2-Euler-eq2}
\sum_{k=0}^{2m}\frac{E_kE_{2m-k}}{2^{2m+1}k!(2m-k)!}=-\frac{1}{(2m)!}E_{2m+1}(0).
\end{equation}
That is,
\begin{equation}\label{2-Euler-eq3}
\sum_{k=0}^{m}\frac{E_{2k}E_{2m-2k}}{2^{2m+1}(2k)!(2m-2k)!}=-\frac{1}{(2m)!}E_{2m+1}(0)
\end{equation}
because $E_{2k+1}=0$ for $k\geq0.$
Applying (\ref{lam-ft-eq}) and (\ref{beta-re}) to the above equality, we obtain the desired result.
\end{proof}

\begin{theorem}\label{Conv3}
Let $m$ be a positive integer. Then we have
$$\sum_{k=0}^m2^{-2k}\zeta(2k)\lambda(2m-2k+2)=0.$$
\end{theorem}
\begin{proof}
Taking into account (\ref{Bernoulli}) and (\ref{euler poly-def}), we have
$$\begin{aligned}
\sum_{n=0}^\infty 2^nB_n\left(\frac x2\right)\frac{t^n}{n!}
&=\frac{2te^{\frac x2(2t)}}{e^{2t}-1} \\
&=\frac{t}{e^t-1}\frac{2e^{xt}}{e^t+1} \\
&=\left(\sum_{n=0}^\infty B_n\frac{t^n}{n!}\right)\left(\sum_{n=0}^\infty E_n(x)\frac{t^n}{n!}\right) \\
&=\sum_{n=0}^\infty\left(\sum_{k=0}^n\binom nkB_kE_{n-k}(x)\right)\frac{t^n}{n!},
\end{aligned}$$
yielding
\begin{equation}\label{Eu-Be-rel}
2^nB_n\left(\frac{x}{2}\right)=\sum_{k=0}^n\binom nkB_kE_{n-k}(x),\quad n\geq0.
\end{equation}
If we set $n=2m+1$ and $x=0$ in (\ref{Eu-Be-rel}) and use $B_{2m+1}=E_{2m}(0)=0 ~(m\geq1),$
then we obtain
\begin{equation}\label{Eu-Be-rel-2}
\sum_{k=0}^m\binom {2m+1}{2k}B_{2k}E_{2m-2k+1}(0)=0,\quad m\geq1.
\end{equation}
Hence, by the well-known relations (\ref{Riemann}) and (\ref{lam-ft-eq}), we obtain
$$\sum_{k=0}^m2^{-2k}\zeta(2k)\lambda(2m-2k+2)=0,\quad m\geq1,$$
as required.
\end{proof}

\section{Power series expansion for $J(s,a)$}
We start with
$$\frac{\Gamma(s)}{(n+a)^s}=\int_0^\infty e^{-(n+a)t}t^{s-1}dt, \quad\text{Re}(s)>0.$$
Applying this to (\ref{J-def}), we obtain
\begin{equation}\label{rem-1}
\begin{aligned}
\Gamma(s)J(s,a)&=\sum_{n=0}^\infty\int_0^\infty (-1)^n e^{-(n+a)t}t^{s-1}dt \\
&=\int_0^\infty\sum_{n=0}^\infty(-1)^n e^{-(n+a)t}t^{s-1}dt \\
&=\int_0^\infty\frac{e^{-at}t^{s-1}}{e^{-t}+1}dt
\end{aligned}
\end{equation}
for Re$(s)>0.$ That is,
\begin{equation}\label{rem-2}
\Gamma(s)J(s,a)=\int_0^\infty\frac{e^{(1-a)t}t^{s-1}}{e^{t}+1}dt, \quad\text{Re}(s)>0,
\end{equation}
(see (\ref{J-int-re})).
By splitting the above integral at $x$ and employing the generating function of the Euler polynomials
\begin{equation}\label{rem-3}
\frac{2e^{zt}}{e^t+1}=\sum_{n=0}^\infty E_n(z)\frac{t^n}{n!},\quad |t|<\pi,
\end{equation}
for the integral over $[0,x)$ obtain
\begin{equation}\label{rem-4}
\begin{aligned}
\Gamma(s)J(s,a)&=\int_x^\infty\frac{e^{(1-a)t}t^{s-1}}{e^{t}+1}dt+\int_0^x\frac{e^{(1-a)t}t^{s-1}}{e^{t}+1}dt  \\
&=\int_x^\infty\frac{e^{(1-a)t}t^{s-1}}{e^{t}+1}dt+\frac12\sum_{n=0}^\infty\frac{E_n(1-a)}{n!}\int_0^xt^{n+s-1}dt.
\end{aligned}
\end{equation}
To handle the first integral in the above equality, we employ a standard integral representation for the incomplete Gamma function
$$\Gamma(s,x)=\int_x^\infty t^{s-1}e^{-t}dt$$
(see, e.g., \cite[p.~260]{AS} and \cite[(1.1)]{Co}), together with a geometric series expansion, to obtain
\begin{equation}\label{rem-5}
\begin{aligned}
\int_x^\infty\frac{e^{(1-a)t}t^{s-1}}{e^{t}+1}dt
&=\sum_{n=0}^\infty(-1)^n\int_x^\infty e^{-(n+a)t}t^{s-1}dt  \\
&=\sum_{n=0}^\infty(-1)^n\frac{1}{(n+a)^s}\Gamma(s,(n+a)x).
\end{aligned}
\end{equation}
From (\ref{rem-4}) and (\ref{rem-5}), we obtain the following expressions of $\Gamma(s)J(s,a).$

\begin{proposition}\label{pro-main}
Let $|x|<\pi$ and {\rm Re}$(s)>0.$ Then we have
$$
\Gamma(s)J(s,a)=\sum_{n=0}^\infty(-1)^{n}\frac{\Gamma(s,(n+a)x)}{(n+a)^s}
+\frac12\sum_{n=0}^\infty\frac{E_n(1-a)}{n!}\frac{x^{n+s}}{n+s}.
$$
The above representation holds in the whole complex plane $\mathbb C.$ Furthermore, it implies the special values of $J(s,a)$ at non-positive integers \cite[(3.8)]{WY} for $k\geq0$ given by
$$
J(-k,a)=\frac{(-1)^k}{2}E_k(1-a)=\frac12 E_k(a).
$$
\end{proposition}

Taking $a=\frac12$ in (\ref{J-def}), it follows from (\ref{beta-def}) that
\begin{equation}\label{beta-val}
J\left(s,\frac12\right)=2^s\sum_{n=0}^\infty\frac{(-1)^{n}}{(2n+1)^s}=2^s\beta(s).
\end{equation}
Thus, the above proposition also implies the following result.

\begin{corollary} We have
$$
\Gamma(s)\beta(s)=\frac1{2^s}\sum_{n=0}^\infty(-1)^{n}\frac{\Gamma(s,(n+1/2)x)}{(n+1/2)^s}
+\frac1{2^{s+1}}\sum_{n=0}^\infty\frac{E_n}{2^{n}n!}\frac{x^{n+s}}{n+s}
$$
with the free parameter $x\in [0,\pi).$ In particular, we obtain the special values
$\beta(-k)=E_k/2$ of $\beta(s)$ at non-positive integers for $k\geq0.$
\end{corollary}

Let $$\eta(s)=(1-2^{1-s})\zeta(s)$$ be the alternating zeta function (in Euler's notation, $M(s)$; see \cite[p. 70]{Varadarajan}). That is,
\begin{equation}\label{eta-def}
\eta(s)=\sum_{n=1}^\infty\frac{(-1)^{n-1}}{n^s},\quad\text{Re}(s) > 0,
\end{equation}
(see \cite[pp. 807--808]{AS}).
Then, taking $a=1$ in (\ref{J-def}), we obtain $J(s,1)=\eta(s).$
The former is sometimes called Lerch's eta function, and the latter is sometimes named Dirichlet's eta function~\cite{Wiki}.
Proposition \ref{pro-main} may have many interesting applications. For example, letting $a=1$, we recover the following result by Coffey.

\begin{corollary}[Coffey, {\cite[p.~1384, Proposition 1]{Co}}]\label{Coffey+}
$$\Gamma(s)\eta(s)=\sum_{n=1}^\infty(-1)^{n-1}\frac{\Gamma(s,nx)}{n^s}
+\frac12\sum_{n=0}^\infty\frac{E_n(0)}{n!}\frac{x^{n+s}}{n+s},
$$
with the free parameter $x\in [0,\pi).$ In particular, we obtain the exact evaluations
$\eta(-k)=(-1)^kE_k(0)/2$ for $k\geq0.$
\end{corollary}

\begin{proposition}\label{pro-main2}
For $0<a\leq1,$ we have
$$\Gamma(s)J(s,a)=\Gamma(s)\frac{(a+1)^{1-s}-a^{1-s}}{2(s-1)}-\sum_{n=1}^\infty\frac{2^n\Gamma(s+n)}{(n+1)!}J(s+n,a). $$
\end{proposition}
\begin{proof}
From the power series expansion of $e^{t},$ we have
\begin{equation}\label{pow-ex1}
\frac1{2t}e^{2t}-\frac1{2t}-1=\sum_{n=1}^\infty\frac{(2t)^n}{(n+1)!}.
\end{equation}
In (\ref{rem-2}), by replacing $s$ by $s+n$, multiplying both sides by  $\frac{2^n}{(n+1)!}$, and summing from $n=1$ to $\infty$, it then follows from (\ref{pow-ex1}) and (\ref{rem-2}) that
\begin{equation}\label{pow-re-int}
\begin{aligned}
\sum_{n=1}^\infty\frac{2^n}{(n+1)!}&\Gamma(s+n)J(s+n,a) \\
&=\int_0^\infty\frac{e^{(1-a)t}t^{s-1}}{e^{t}+1}\sum_{n=1}^\infty\frac{(2t)^n}{(n+1)!}dt \\
&=\int_0^\infty\frac{e^{(1-a)t}t^{s-1}}{e^{t}+1}\left(\frac1{2t}e^{2t}-\frac1{2t}-1\right) dt \\
&=\frac{\Gamma(s-1)}{2}(J(s-1,a+2)-J(s-1,a))-\Gamma(s)J(s,a) \\
&=\frac{\Gamma(s-1)}{2}((a+1)^{1-s}-a^{1-s})-\Gamma(s)J(s,a).
\end{aligned}
\end{equation}
The final equality holds because $J(s,a+1)+J(s,a)=a^{-s}.$
Therefore, the proposition is proven.
\end{proof}

\section*{Acknowledgement} The authors would like to thank the referee for valuable comments and suggestions.

\bibliography{central}

\end{document}